\documentclass[11pt]{article}
\usepackage{amssymb,latexsym,amsfonts,verbatim,amscd}
\usepackage{amsmath}
\usepackage[all]{xy}

\newtheorem{Def}{Definition}[section]
\newtheorem{Thm}[Def]{Theorem}
\newtheorem{Lem}[Def]{Lemma}
\newtheorem{Prop}[Def]{Proposition}
\newtheorem{Ex}[Def]{Example}
\newtheorem{Cor}[Def]{Corollary}
\newtheorem{Fac}[Def]{Fact}

\font\nat msbm10 scaled\magstephalf
\def\N{\hbox{\nat\char78}}

\font\mata=msam10 
\def\restr{\mbox{\mata\char22}}
\def\telos{\hfill$\dashv$}

\begin{document}

\title{The magmatic universe revisited: we define ordered pairs, relations, numbers and  a special form of Separation}

\author{Athanassios Tzouvaras}

\date{}
\maketitle

\begin{center}
Department  of Mathematics\\  Aristotle University of Thessaloniki \\
541 24 Thessaloniki, Greece \\
 e-mail: \verb"tzouvara@math.auth.gr"
\end{center}

\begin{abstract}
This is a companion article to \cite{Tz24}. We address the following two questions: 1) Can we define  in the magmatic universe $M$ of \cite{Tz24} counterparts, or just analogues, of some very basic set-theoretic objects which are missing from $M$, specifically  ordered pairs, binary relations, especially functions, as well as  natural and ordinal  numbers? 2) Are there restricted  forms of the Separation and, perhaps,  Replacement schemes that hold in $M$? We show the following: 1) Magmatic analogues of  ordered pairs can indeed be defined by means of certain magmas called ``magmatic pairs''. However when we use them to generate relations and especially functions, some  unsurmountable problems come up. These problems are due to the peculiarity of the elements of magmas to be distinguished into ``intended'' and ``collateral'' ones, a distinction due to their inherent relation of dependence. So magmatic functions are defined under very special conditions. 2) A certain class of formulas, called ``magmatic formulas'' is isolated, and  the scheme of Separation  restricted to these formulas, called ``Magmatic Separation Scheme'' (MSS), is proven to hold in $M$. On the other hand Replacement fails badly, and this is due to its functional form.

\end{abstract}


\vskip 0.2in


\section{Introduction}

\subsection{The content of the paper}

The magmatic universe $M$ that was constructed in \cite{Tz24} as a subuniverse of $V(A)$, the universe of ZF with a set of atoms $A$,  in order to capture  formally some aspects  Castoriadis' notion of magmas as exposed in \cite{Ca87} and \cite{Ca89}, has significant differences from the universe of sets. An important such difference is that it contains {\em no finite} sets (including the  empty set). As a consequence $M$ does not contain usual ordered pairs $\langle x,y\rangle:=\{\{x\},\{x,y\}\}$, and hence cartesian products as well as relations (binary and $n$-ary), and in particular functions, since all these  mathematical entities  are constructed exclusively by the help  of ordered pairs. So a natural technical challenge would be to find out in $M$ magmatic analogues, or magmatic counterparts,   of the aforementioned  basic objects, primarily of ordered pairs, and then by the use of them magmatic products, relations,  functions and natural and ordinal numbers.  That would help us to explore further the universe $M$, and  understand what is like to live not in the familiar set-theoretic ``separative'' world  of independent entities, but  in a ``non-separative'' world of dependent inseparable entities.

In this companion paper to \cite{Tz24} we set out to address  mainly this project. The starting point is the observation that the magmatic analogues of the singletons   $\{x\}$,  for $x\in M$, are the basic open sets $pr(x)$ of $M$. This allows us to define for any $x,y\in M$, a magma $\langle\langle x,y\rangle\rangle\in M$ which satisfies the basic condition for pairs, namely  $\langle\langle x,y\rangle\rangle=\langle\langle x',y'\rangle\rangle$ if and only if $x=x'$ and $y=y'$. $\langle\langle x,y\rangle\rangle$ is the magmatic ordered pair of $x,y$. However, in contrast to set-theoretic pairs, for any pair $\langle\langle x,y\rangle\rangle$ there are infinitely many sub-pairs $\langle\langle x',y'\rangle\rangle\subseteq \langle\langle x,y\rangle\rangle$. This causes serious complications to the definition of products, and especially relations and functions, as magmas consisting of magmatic pairs, and leads to the distinction of elements of a magma into ``intended'' and ``collateral'' ones.  Namely, whenever we want a relation $\textbf{R}$ to contain some specific pair   $\langle\langle x,y\rangle\rangle$, to which we refer to as ``intended'', as a consequence of dependence  $\textbf{R}$ must contain also every $\langle\langle x',y'\rangle\rangle\subseteq \langle\langle x,y\rangle\rangle$ (i.e. every pair depended to $\langle\langle x,y\rangle\rangle$, since $\subseteq$ is the relation that captures dependence in $M$.) We refer to every such pair $\langle\langle x',y'\rangle\rangle$ which depends on the intended $\langle\langle x,y\rangle\rangle$ as a  ``collateral'' one. The complication and  non-perspicuousness in the  definition of relations is due exactly to the occurrence of a plethora of collateral pairs along  with any intended one. The problem becomes more   acute when we come to define magmatic functions, that is  magmatic relations with the usual unique output  condition for each input. Such relations do not exist in $M$ with the present definition of magmatic pair, except only under a modified definition and special conditions. A solution would be to invent a  definition of $\langle\langle x,y \rangle\rangle$ in $M$ according to which   there are no proper collateral sub-pairs.  But I don't believe that  such a definition is  attainable.

Further we show how natural and ordinal numbers can be defined in $M$, and also that the Separation scheme holds in $M$ for a special class of formulas called ``magmatic''.

The content of the paper is as follows:

In section 2 we recall some technical facts from \cite{Tz24} that will be important for the rest of the paper.

In section 3 we introduce the  magmatic ordered pairs and prove the basic facts about them.

In section 4 we deal with magmatic products, which are the analogues of cartesian  products, and for the first time we come across the crucial role played by the  ``collateral'' elements in distinction of the ``intended'' elements of magmas. Moreover here we introduce the  magmatic  (binary) relations and functions, in particular, and discuss the problems we face  with them.

In section 5 we define magmatic natural and ordinal  functions, and give some examples of how can be used in combination with functions.

In section 6 we show that a natural restricted form of the Separation Scheme, the Magmatic Separation Scheme (MSS), that employs only  ``magmatic'' formulas is true in $M$. On the other hand, it is shown that the Replacement axiom fails irreparably, as a consequence of its functional character.

\subsection{A brief overview of \cite{Tz24}}

In the rest of this Introduction we give, for the reader's convenience,  a brief  overview of magmas and the magmatic universe as presented in \cite{Tz24}. Let ${\rm ZFA}$ be the theory  ${\rm ZF}$ with a set $A$ of  atoms  (or urelements) and let $V(A)$  denote its universe built on $A$. In \cite{Tz24} we introduced and studied the ``magmatic universe'' $M(A)$ (or, more precisely, $M(A,\preccurlyeq)$, since the definition of $M(A)$ depends heavily on a preordering $\preccurlyeq$ of the set $A$  of atoms). In order for $\preccurlyeq$ to simulate more faithfully the dependence relation that occurs  amongst elements of collections which C. Castoriadis  \cite{Ca87} and \cite{Ca89} refers to   as ``magmas'', we assume that $\preccurlyeq$ has the  extra property not to contain minimal elements, that is for every $a\in A$ there is a $b$ such that $b\prec a$ (which means $b\preccurlyeq a$ and $a\not\preccurlyeq b$).

Given the preordering $\preccurlyeq$ of $A$, for every $a\in A$ let $pr(a)=\{b\in A:b\preccurlyeq a\}$ be the set of predecessors of $a$. The sets $pr(a)$ are the basis of a natural topology on $A$, called the {\em lower topology}, that is, a $x\subseteq A$ is ``lower open'', or just open, if for every $a\in x$, $pr(a)\subseteq x$.  Let $LO(A,\preccurlyeq)$ be the set of lower open subsets of $A$. Due to the non-minimality condition imposed to $\preccurlyeq$, it follows easily that (a) all sets in $LO(A,\preccurlyeq)$ are infinite,  and (b) there is no $\subseteq$-minimal set in $LO(A,\preccurlyeq)$. The elements of $LO(A,\preccurlyeq)$ constitute the magmas of the first (bottom) level of the universe of magmas. For we can shift the preordering $\preccurlyeq$ to the powerset of $A$, and then, successively, to all higher powersets as follows. For $x,y\subseteq A$, let
$$x \preccurlyeq^+y:\Leftrightarrow (\forall a\in x)(\exists b\in y)(a \preccurlyeq b).$$
It is easy to see that $\preccurlyeq^+$ is a preordering on ${\cal P}(A)$. Moreover if $\preccurlyeq$ does not contain minimal elements, then so does $\preccurlyeq^+$. Now if  we  let
$$pr^+(x)=\{y\subseteq A: y\preccurlyeq^+ x\},$$
then the following remarkable relation holds between the sets $pr^+(x)$ and the powersets of $x$.

\begin{Prop} \label{P:connection}
{\rm (\cite[Prop. 3.2]{Tz24})}

(i) For any $x,y\in {\cal P}(A)$, $y\subseteq x \Rightarrow  y\preccurlyeq^+ x$, so  ${\cal P}(x)\subseteq pr^+(x)$.

(ii) If $x\in LO(A,\preccurlyeq)$, then the converse of (i) holds, i.e., for every $y\in {\cal P}(A)$, $y\preccurlyeq^+ x \Rightarrow  y\subseteq x$, so  $pr^+(x)\subseteq {\cal P}(x)$.

(iii) Therefore for every  $x\in LO(A,\preccurlyeq)$,   $pr^+(x)={\cal P}(x)$.

(iv) In particular, $\preccurlyeq^+\restr LO(A,\preccurlyeq)=\subseteq$, and
$$LO(LO(A,\preccurlyeq), \preccurlyeq^+)=LO(LO(A,\preccurlyeq), \subseteq).$$
\end{Prop}

It follows from clause (iv) above that if we set $M_1=LO(A,\preccurlyeq)$, then $LO(M_1,\preccurlyeq^+)=LO(M_1,\subseteq)$. And for the same reason,  if we set $M_2=LO(M_1,\subseteq)$, the  shifting $\subseteq^+$ of the relation $\subseteq$ of $M_1$ to the sets of $M_2$ is $\subseteq$ again.  And so on with $\subseteq^{++}$, $\subseteq^{+++}$ etc,  for all subsequent levels $M_3$, $M_4,\ldots$, which are defined similarly. Moreover, if $\alpha$ is  a limit ordinal, we can take as $M_\alpha$ just the union of all previous  levels, ordered again by $\subseteq$. And then we can proceed further by setting   $M_{\alpha+1}=LO(M_\alpha,\subseteq)$. This way   $M_\alpha$ is defined for every ordinal $\alpha\geq 1$.

In view of the above facts, the magmatic universe  $M(A)$ is defined in $V(A)$ as follows.

\vskip 0.1in

$M_1=LO(A,\preccurlyeq)$.

$M_{\alpha+1}=LO(M_\alpha,\subseteq)$, for every $\alpha\geq 1$.

$M_\alpha=\bigcup_{1\leq \beta<\alpha}M_\beta$, if $\alpha$  is a limit ordinal.

$M=M(A)=\bigcup_{\alpha\geq 1}M_\alpha$.

\vskip 0.1in

The next Proposition gives  some basic  facts about $M_\alpha$'s and $M$.

\begin{Prop} \label{P:basics}
{\rm (\cite[Lemma 4.2]{Tz24})}

(i) $M_\alpha\subseteq V_\alpha(A)$, $M_1\subseteq {\cal P}(A)\backslash \{\emptyset\}$,  $M_{\alpha+1}\subseteq {\cal P}(M_\alpha)\backslash \{\emptyset\}$ and  $M_\alpha\in M_{\alpha+1}$, for every $\alpha\geq 1$.

(ii) The b.o. sets of the space $M_{\alpha+1}=LO(M_\alpha,\subseteq)$, for $\alpha\geq 1$,  are the sets
$$pr_\alpha(x)=\{y\in M_\alpha:y\subseteq  x\}={\cal P}(x)\cap M_\alpha.$$
Therefore:
$$x\in M_{\alpha+1} \ \Leftrightarrow \ x\subseteq M_\alpha \wedge x\neq\emptyset \wedge (\forall y\in x)({\cal P}(y)\cap M_\alpha\subseteq x).$$

(iii) The class $M$ is ``almost'' transitive, in the sense that for every $x\in M_\alpha$ for $\alpha\geq 2$, $x\subseteq M$. However for $x\in M_1$,  $x\subseteq A$. So $M\cup A$ is transitive.

(iv) $M$ is a proper subclass of $V(A)$, while $M\cap V=\emptyset$, where $V$ is  the subclass of pure sets of $V(A)$ (that is of $x$'s such that $TC(x)\cap A=\emptyset$).

(v) All sets of $M$ are infinite.

(vi) There is no $\subseteq$-minimal set in $M$.
\end{Prop}

In \cite{Tz24} magmas were introduced primarily in order to address Castoriadis' concerns about collections that, seemingly, cannot be accommodated within the standard cantorian universe, and also capture his intuitions about basic properties of such collections. Nevertheless the magmatic hierarchy/universe $M=M(A)=\bigcup_{\alpha\geq 1}M_\alpha$  seems to have a more general interest, or an interest in its own, in the following sense. It constitutes a hierarchical totality of collections which differs from the corresponding set-universe $V(A)$ only with respect to the {\em inseparability} of its elements, that is, their inability to split to arbitrary sub-collections. Of course this causes the failure of many axioms of ${\rm ZFC}$ (Pairing, Union, Separation, etc), since the purpose of most of them is exactly to ensure the existential independence of sets and their  greatest possible diversity.

\section{The consequences  of dependence on the behavior of magmas}
The dependence of elements has a strong impact on magmas and it is what makes their behavior   diverge significantly from that of sets. The main point of  divergence is the lack of {\em perspicuity} and {\em determinateness} concerning magmas' elements. Practically this means that we can never be fully aware of what {\em exactly} a magma contains. For example  given any prescribed magmas $x_1,\ldots,x_n$, we can  define other  magmas that contain them, and moreover  a {\em smallest} one, namely $pr(x_1)\cup\cdots\cup pr(x_n)$, however even this collection never contains {\em exactly} $x_1,\ldots,x_n$. It contains also an infinity  of other magmas, all being submagmas of $x_i$, for the properties of which  we cannot tell anything precise. This is in sharp contrast to what happens with sets. One of the fundamental mechanisms there for  set construction is the Separation axiom, which states  that for every property $\phi(x)$ and every pre-existed set $X$ there is set $Y$ which contains {\em exactly} the elements of $X$ that satisfy $\phi(x)$.   Assuming that we have a clear idea what the  objects satisfying  $\phi(x)$ look like, as well as what the objects of $X$ are, it follows that the elements of $Y$ are fully determined. Nothing analogous  holds for $M$. The basic properties $\phi(x,x_1,\ldots,x_n)$ of $L=\{\in\}$, with parameters $x_1,\ldots,x_n$, which define magmas inside $M$ are those defining $pr(x_1)\cup\cdots\cup pr(x_n)=({\cal P}(x_1)\cup\ldots\cup {\cal P}(x_n))\cap M$, that is
$$\phi(x,x_1,\ldots,x_n):=\bigvee_{i=1}^nx\subseteq x_i.\footnote{In connection with the lack of perspicuity and unclarity of the elements of magmas mentioned above,   I find it quite remarkable and intriguing  that the formulas $\phi(x,x_1,\ldots,x_n)$ define in $M$ {\em infinite powersets} (from the point of view of ${\rm ZF}$), since powersets of infinite sets have notoriously been characterized  as ``inherently vague'' collections by many researchers (among them  S. Feferman, N. Weaver, H. Field, and others).  }$$
Nevertheless, the indeterminateness and unclarity of element-hood of magmas, as they are represented in the  formal framework of \cite{Tz24},  shows  that the proposed formalization was rather successful. In fact  the reason that  Castoriadis developed and elaborated  the concept and ``logic'' of magmas was exactly to bring to fore  qualities of thinking (logic) and collections which lie at the antipodes of  full determination and full  clarity, qualities which, he believed, are lurking everywhere around, but cannot be captured and understood by what he called ``Western thought''.

Therefore it should come as no surprise the fact that the objects of $M$ behave so differently from those of a model of ${\rm ZF}$. The question is whether there are any interesting {\em mathematical} facts that happen inside $M$. We already talked about  the general unclarity that characterizes elements of magmas. However the problem becomes much more acute if we attempt to define {\em relations} in $M$, especially functions. As is well-known the only so far widely accepted method to define mathematically these entities is by means of {\em ordered pairs}. So a first question is whether ``magmatic pairs'', as analogues of the usual ones, can be naturally defined in $M$. The answer is yes. We shall see in the next section that we can  define {\em magmatic analogues} of ordered pairs in a satisfactory way, namely magmas $\langle\langle x,y\rangle\rangle$ for any $x,y\in M$, having the basic property of pairs, that is,
$$\langle\langle x,y\rangle\rangle=\langle\langle x',y'\rangle\rangle\Rightarrow x=x' \ \& \ y=y'.$$
Nevertheless we shall also see that this is not sufficient in order to define faithfully  binary relations, and especially functions,  imitating the usual definitions of the corresponding ${\rm ZF}$-objects.

\subsection{Some preliminary technical facts about magmas}
The  basic idea behind the  construction in $M$ of  magmatic counterparts of standard ${\rm  ZFA}$ objects, like ordered pairs, is to use in place of singletons $\{x\}$ and, more generally, finite sets $\{x_1,\ldots,x_n\}$, which do not exist in $M$,  the corresponding {\em smallest} magmas that contain them.  For example, for every  $a\in A$  the smallest magma $u$ such that $a\in u$ is the b.o. set $pr(a)$ of predecessors of $a$, and more generally, for any finite $\{a_1,\ldots,a_n\}\subseteq A$, the smallest magma that contains all $a_i$ is $pr(a_1)\cup\cdots\cup pr(a_n)$. However, a  problem seems to be, as follows from Proposition \ref{P:basics} (ii), that for each $x\in M$ there is an array $pr_\alpha(x)$ of such elementary sets, one for each $\alpha$ such that $x\in M_\alpha$. (Notice that that's not the case for atoms  $a\in A$, since $A\cap M=\emptyset$). However all these b.o. sets are {\em identical}, as follows from Corollary \ref{C:identical} below.

\begin{Prop} \label{P:powrank}
{\rm (\cite{Tz24}, Prop. 4.5, Cor. 4.6)}

(i) If $x\in M_1$, then ${\cal P}(x)\cap M\subseteq M_1$, and hence ${\cal P}(x)\cap M={\cal P}(x)\cap M_1$.

(ii) If $x\in M_{\alpha+1}$, then ${\cal P}(x)\cap M\subseteq M_{\alpha+1}$, for every $\alpha\geq 1$, and hence  ${\cal P}(x)\cap M={\cal P}(x)\cap M_{\alpha+1}=pr_{\alpha+1}(x)$.

(iii) If $\alpha<\beta$ are limit ordinals and $x\in M_\alpha$, then $pr_\alpha(x)=pr_\beta(x)$.

(iv) For every $\alpha\geq 0$, if $x\in M_{\alpha+1}$ then ${\cal P}(x)\cap M\in M_{\alpha+2}$.

(v) If $x\in M$ and $x\subseteq M_\alpha$, then $x\in M_{\alpha+1}$  (that is,  every subset of $M_\alpha$ that belongs to $M$, is an open subset of $M_\alpha$).
\end{Prop}

Let  $rank_M(x)$ denote the {\em magmatic rank} of $x$, i.e., $rank_M(x)=\min\{\alpha:x\in M_{\alpha+1}\}$.

\begin{Cor} \label{C:identical}
Let $x\in M$ and let $rank_M(x)=\alpha$. Then for every $\beta\geq \alpha+1$ such that $x\in M_\beta$, $pr_\beta(x)=pr_{\alpha+1}(x)={\cal P}(x)\cap M$.
\end{Cor}

{\em Proof.} Let $\beta\geq \alpha+1$ and $x\in M_\beta$. If $\beta$ is successor, say $\beta=\gamma+1$, then by clause (ii) of \ref{P:powrank}, $pr_\beta(x)=pr_{\gamma+1}(x)={\cal P}(x)\cap M_{\gamma+1}={\cal P}(x)\cap M$. But also for the same reason $pr_{\alpha+1}(x)={\cal P}(x)\cap M$, hence $pr_\beta(x)=pr_{\alpha+1}(x)$. If now $\beta$ is limit, then $M_{\alpha+1}\subseteq M_\beta$, therefore $pr_{\alpha+1}(x)\subseteq pr_\beta(x)={\cal P}(x)\cap M_\beta\subseteq {\cal P}(x)\cap M$. But already $pr_{\alpha+1}(x)={\cal P}(x)\cap M$, so $pr_\beta(x)=pr_{\alpha+1}(x)={\cal P}(x)\cap M$. \telos

\vskip 0.2in

In view of Corollary \ref{C:identical}, we henceforth drop, as a rule,  the subscripts from $pr_\alpha(x)$ and write simply $pr(x)$ for every  $x\in M$, where $pr(x)$ always denotes $pr_{rank(x)+1}(x)$, or according to \ref{C:identical},
\begin{equation} \label{E:relPow}
pr(x)={\cal P}(x)\cap M.
\end{equation}
Also below we shall often use iterates of the operation $pr(\cdot)$, writing $pr^2(x)$ instead of $pr(pr(x))$, and in general $pr^n(x)$, $n\in \N$, for the $n$-th iterate.

\begin{Lem} \label{L:injective}
The following hold for the operation $pr:A\cup M\rightarrow M$:

(i) $pr$ is injective on $M$, i.e., $pr(x)=pr(y) \Rightarrow x=y$, for all $x,y\in M$.

(ii) $pr$ is an embedding  with respect to $\subseteq$,  i.e., $x\subseteq y \Leftrightarrow pr(x)\subseteq pr(y)$.

(iii) For any $a\in A$ and $x\in M$, $pr(a)\neq pr(x)$.

(iv) For $a,b\in A$, $pr(a)=pr(b)\Leftrightarrow a\sim b$, where $a\sim b$ denotes the equivalence $a\preccurlyeq b \wedge b\preccurlyeq a$.
\end{Lem}

{\em Proof.} (i) Let $x,y\in M$. By (\ref{E:relPow}) $pr(x)=pr(y)$ means ${\cal P}(x)\cap M={\cal P}(y)\cap M$, and hence $x\in {\cal P}(y)$ and $y\in {\cal P}(x)$. It follows  $x\subseteq y$ and $y\subseteq x$, that is $x=y$.

(ii) If $x\subseteq y$, by  (\ref{E:relPow}) we have  immediately that  $pr(x)\subseteq pr(y)$.  Conversely, if $pr(x)\subseteq pr(y)$ then $x\in pr(y)$ since $x\in pr(x)$, hence $x\subseteq y$.

(iii) If $a\in A$, $pr(a)\subseteq A$, while for $x\in M$, $pr(x)\subseteq M$. Since $A\cap M=\emptyset$, the claim follows.

(iv) This follows immediately from the definition of $pr(a)$.  \telos

\vskip 0.2in

Below we shall define magmatic analogues of ordered pairs, as well as natural and ordinal numbers. In these constructions we largely use  {\em unions} of magmas so we must  make sure which of such unions are legitimate, i.e., are elements of $M$, and which are  not. The following Fact  clarifies the situation.

\begin{Lem} \label{L:legitim}
Let $x, y\in M$.

(i) If $x,y\in M_1$, then $x\cup y\in M_1$.

(ii) If $x,y\in M\backslash M_1$, then $x\cup y\in M\backslash M_1$.

(iii) If $x\in M_1$ and $y\in M\backslash M_1$, then $x\cup y\notin M$.
\end{Lem}

{\em Proof.} (i) Recall that $M_1=LO(A,\preccurlyeq)$, i.e. $M_1$  consists of the lower open subsets of $A$, hence it is closed under unions.

(ii) Let $x,y\in M\backslash M_1$. Then $x\in M_{\alpha+1}$ and  $y\in M_{\beta+1}$ for some $\alpha,\beta\geq 1$. Therefore $x\subseteq M_\alpha$ and $y\subseteq M_\beta$. Then for any limit ordinal $\gamma\geq \alpha,\beta$ we have  $x\cup y\subseteq M_\gamma$ and thus $x\cup y\in M_{\gamma+1}$. So obviously $x\cup y\notin M_1$.

(iii) Let $x\in M_1$ and $y\in M\backslash M_1$. If $x\cup y\in M$, then clearly $x\cup y\notin M_1$ (otherwise $y\in M_1$). So $x\cup y\in M_{\alpha+1}$ for some $\alpha\geq 1$, and hence  $x\cup y\subseteq M_\alpha\subseteq M$. But this  is false because $x\subseteq A$ as an element of $M_1$, and $A\cap M=\emptyset$. \telos

\vskip 0.2in

An important consequence of clause (iii) of the preceding Lemma is that  in the codings that we shall use in the next section for the definition of magmatic pair, we cannot use unions  of the form $pr(a)\cup pr(pr(a))$, for $a\in A$, since $pr(a)\in M_1$ while $pr(pr(a))\in M\backslash M_1$, and thus, according to (iii), $pr(a)\cup pr(pr(a))\notin M$. So instead of $pr(a)$ we shall use  $pr^2(a)=pr(pr(a))$ which  belongs to  $M_2$.

\section{Magmatic ordered pairs}
Various definitions of ordered pair have appeared  in the literature of set theory along the time, with the prevailing and  widely adapted one being  Kuratowski's definition of $\langle x,y\rangle$ as the set $\{\{x\},\{x,y\}\}$. See \cite{Wiki}, and the sources cited there, for an overview of all proposed definitions. Among the listed definitions is also the following: $\langle x,y\rangle:=\{\{x,0\},\{y,1\}\}$. Its only ``defect'', in comparison  to Kuratowski's pair, is that it presumes the  existence of the ``special'' objects $0$ and $1$. The definition of magmatic pair we give below  resembles more the definition of pair based on the use of the special objects $0,1$, than Kuratowski's one. Specifically we fix two atoms $a_0,a_1\in A$ which are incomparable with respect to $\preccurlyeq$, i.e., $a_0\not\preccurlyeq a_1$ and $a_1\not\preccurlyeq a_0$, so equivalently $pr(a_0)\not\subseteq pr(a_1)$ and $pr(a_1)\not\subseteq pr(a_0)$. The reason for such a choice will be clear later.

\begin{Def} \label{D:MagmaPair}
{\em For any $x,y\in M$, the} magmatic pair {\em with first constituent $x$ and second constituent $y$, denoted $\langle\langle x,y\rangle\rangle$, is defined as follows:
$$\langle\langle x,y\rangle\rangle:=pr(pr^2(x)\cup pr^2(a_0))\cup pr(pr^2(y)\cup pr^2(a_1)).$$
}
\end{Def}

\textbf{Remarks.} (1) In the preceding definition  the parts  $pr^2(x)\cup pr^2(a_0)$ and $pr^2(y)\cup pr^2(a_1)$ imitate  the sets $\{x,0\}$ and $\{y,1\}$, respectively, while $\langle\langle x,y\rangle\rangle$ imitates the pair $\{\{x,0\}, \{y,1\}\}$.

(2) The reason for  raising  the ranks of   $pr(a_0)$, $pr(a_1)$ and replacing them by $pr^2(a_0)$, $pr^2(a_1)$, is due to Lemma \ref{L:legitim} (iii), according to which for $a\in A$ $pr(a)\cup x$ is not a magma unless $x\in M_1$.

(3) The use of $pr^2(a_0)$ and $pr^2(a_1)$ made  it necessary, for reasons of symmetry, to use in turn $pr^2(x)$ and $pr^2(y)$ instead  of $pr(x)$ and  $pr(y)$,  respectively.

\vskip 0.1in

The main  task  is to verify that the objects   $\langle\langle x,y\rangle\rangle$ have indeed the intended  property of ordered pairs, namely $\langle\langle x,y\rangle\rangle=\langle\langle x',y'\rangle\rangle$ $\Rightarrow$ $x=x'$ and $y=y'$. In order to establish that we need some preliminary lemmas.

\begin{Lem} \label{L:two-one}
Let $x,y,x',y',z\in M$. Then:

(i) $pr(x)\cup pr(y)\subseteq pr(z)$ iff $x\subseteq z$ and $y\subseteq z$.

(ii)  $pr(z)\subseteq pr(x)\cup pr(y)$ iff $z\subseteq x$ or $z\subseteq y$.

(iii)  $pr(x)\cup pr(y)=pr(z)$ iff  $x\subseteq y=z$ or $y\subseteq x=z$.

(iv) $pr(x)\cup pr(y)\subseteq pr(x')\cup pr(y')$ iff one of the following hold:   (a) $x,y\subseteq x'$, (b)  $x,y\subseteq y'$, (c)  $x\subseteq x'$ and $y\subseteq y'$, (d) $x\subseteq y'$ and $y\subseteq x'$.
\end{Lem}

{\em Proof.} (i)  $pr(x)\cup pr(y)\subseteq pr(z)$ iff $x\in pr(z)$ and $y\in pr(z)$ iff  $x\subseteq z$ and $y\subseteq z$.

(ii) $pr(z)\subseteq pr(x)\cup pr(y)$ iff $z\in pr(x)$ or $z\in pr(y)$ iff  $z\subseteq x$ or $z\subseteq y$.

(iii) $pr(x)\cup pr(y)=pr(z)$ iff the conjunction of (i) and (ii) holds, i.e. iff
$x\subseteq z\wedge y\subseteq z\wedge z\subseteq x$, or $x\subseteq z\wedge y\subseteq z\wedge z\subseteq y$. The first conjunction is shorty written  $y\subseteq x=z$ and the second one is shortly written  $x\subseteq y=z$.

(iv) Let $pr(x)\cup pr(y)\subseteq pr(x')\cup pr(y')$. Clearly the inclusion holds iff $x,y\in pr(x')$, or  $x,y\in pr(y')$, or ($x\in  pr(x')$ and $y\in  pr(y')$), or  ($x\in pr(y')$ and $y\in  pr(x')$). These four possibilities are equivalently written as possibilities  (a), (b), (c), (d), respectively. \telos

\vskip 0.2in

\begin{Lem} \label{L:union-of-bo}
Let $x,y$, $x',y'$ in $M$. Then   $$pr(x)\cup pr(y)=pr(x')\cup pr(y'),$$
if and only if  one of the following holds:

(I) $x=x'$ and $y=y'$, or

(II) $x=y'$ and $y=x'$, or

(III) One of the elements of $\{x,y\}$ equals one  of the elements of $\{x',y'\}$, while the  other two are submagmas  of the equal ones.
\end{Lem}

{\em Proof.} Let $x,y,x',y'\in M$ and assume  $pr(x)\cup pr(y)=pr(x')\cup pr(y')$. This is equivalent to the conjunction of  $pr(x)\cup pr(y)\subseteq pr(x')\cup pr(y')$ and  $pr(x')\cup pr(y')\subseteq pr(x)\cup pr(y)$. By clause (iv) of Lemma \ref{L:two-one}, the first conjunct is equivalent to the disjunction of

(a)  $x,y\subseteq x'$.

(b)  $x,y\subseteq y'$.

(c)  $x\subseteq x'$ and $y\subseteq y'$.

(d)  $x\subseteq y'$ and $y\subseteq x'$.

\noindent Similarly the second  conjunct is equivalent to the disjunction of

(a$'$) $x',y'\subseteq x$.

(b$'$) $x',y'\subseteq y$.

(c$'$) $x'\subseteq x$ and $y'\subseteq y$.

(d$'$) $x'\subseteq y$ and $y'\subseteq x$.

\noindent Therefore the equality $pr(x)\cup pr(y)=pr(x')\cup pr(y')$ is equivalent to the disjunction of the 16  combinations of (a)-(d) with (a$'$)-(d$'$), i.e., to the disjunction of the following conjuncts:

(a) $\&$ (a$'$):   $x,y\subseteq x'$ and $x',y'\subseteq x$. Then
$y,y'\subseteq x=x'$.

(a) $\&$ (b$'$): $x,y\subseteq y'$ and  $x',y'\subseteq y$. Then $x,x'\subseteq y=y'$.

(a) $\&$ (c$'$): $x,y\subseteq x'$, $x'\subseteq x$ and $y'\subseteq y$. Then $y,y'\subseteq x=x'$.

(a) $\&$ (d$'$): $x,y\subseteq x'$, $x'\subseteq y$ and $y'\subseteq x$. Then $x,y'\subseteq x'=y$.

(b) $\&$ (a$'$): $x,y\subseteq y'$,   $x',y'\subseteq x$. Then $x',y\subseteq x=y'$.

(b) $\&$ (b$'$): $x,y\subseteq y'$, $x',y'\subseteq y$. Then $x,x'\subseteq y=y'$.

(b) $\&$ (c$'$):  $x,y\subseteq y'$, $x'\subseteq x$ and $y'\subseteq y$. Then $x,x'\subseteq y=y'$

(b) $\&$ (d$'$): $x,y\subseteq y'$,  $x'\subseteq y$ and $y'\subseteq x$. Then $x',y\subseteq x=y'$.

(c) $\&$ (a$'$): $x\subseteq x'$,  $y\subseteq y'$ and  $x',y'\subseteq x$. Then $y,y'\subseteq x=x'$.

(c) $\&$ (b$'$):  $x\subseteq x'$,  $y\subseteq y'$ and $x',y'\subseteq y$. Then $x,x'\subseteq y=y'$.

(c) $\&$ (c$'$): $x\subseteq x'$,  $y\subseteq y'$, $x'\subseteq x$ and $y'\subseteq y$. Then \underline{$x=x'$ and $y=y'$}.

(c) $\&$ (d$'$): $x\subseteq x'$,  $y\subseteq y'$, $x'\subseteq y$ and $y'\subseteq x$. Then \underline{$x=y=x'=y'$}.

(d) $\&$ (a$'$): $x\subseteq y'$,  $y\subseteq x'$ and $x',y'\subseteq x$. Then $x',y\subseteq x=y'$.

(d) $\&$ (b$'$): $x\subseteq y'$,  $y\subseteq x'$ and $x',y'\subseteq y$. Then $x,y'\subseteq x'=y$.

(d) $\&$ (c$'$): $x\subseteq y'$,  $y\subseteq x'$, $x'\subseteq x$ and $y'\subseteq y$. Then \underline{$x=y=x'=y'$}.

(d) $\&$ (d$'$): $x\subseteq y'$,  $y\subseteq x'$, $x'\subseteq y$ and $y'\subseteq x$. Then \underline{$x=y'$ and $x'=y$}.

\noindent The underlined conjuncts  verify  possibilities (I) and (II), while the rest cases verify possibility (III) of the claim. \telos

\begin{Thm} \label{T:pairs}
For any $x,y$, $x',y'$ in $M$,
$$\langle\langle x,y\rangle\rangle=\langle\langle x',y'\rangle\rangle\Leftrightarrow x=x' \ \& \ y=y'.$$
\end{Thm}

{\rm Proof.}  Suppose $\langle\langle x,y\rangle\rangle=\langle\langle x',y'\rangle\rangle$. Then
$$pr(pr^2(x)\cup pr^2(a_0))\cup pr(pr^2(y)\cup pr^2(a_1))=$$
\begin{equation} \label{E:assume}
pr(pr^2(x')\cup pr^2(a_0))\cup pr(pr^2(y')\cup pr^2(a_1)).
\end{equation}

By Lemma \ref{L:union-of-bo}, equality  (\ref{E:assume}) is true if and only if one of the following is the case:

(I) $pr^2(x)\cup pr^2(a_0)=pr^2(x')\cup pr^2(a_0)$ and $pr^2(y)\cup pr^2(a_1)=pr^2(y')\cup pr^2(a_1)$, or

(II) $pr^2(x)\cup pr^2(a_0)=pr^2(y')\cup pr^2(a_1)$ and $pr^2(y)\cup pr^2(a_1)=pr^2(x')\cup pr^2(a_0)$, or

(III) One of the magmas in $\{pr^2(x)\cup pr^2(a_0), pr^2(y)\cup pr^2(a_1)\}$ is equal to one of the magmas in  $\{pr^2(x')\cup pr^2(a_0), pr^2(y')\cup pr^2(a_1)\}$, and the other two are submagmas of the equal ones.

We check separately each of the cases I, II and III.

\vskip 0.1in

{\em Case I.} Let $pr^2(x)\cup pr^2(a_0)=pr^2(x')\cup pr^2(a_0)$ and $pr^2(y)\cup pr^2(a_1)=pr^2(y')\cup pr^2(a_1)$. We examine the first equality. We shall show that $x=x'$. Since $pr(x)\in pr^2(x)$ (as for every $x$, $x\in pr(x)$), it follows that  $pr(x)\in pr^2(x')\cup pr^2(a_0)$, i.e., $pr(x)\in pr^2(x')$  or $pr(x)\in pr^2(a_0)$. The second option is impossible because it implies  $pr(x)\subseteq pr(a_0)$ (as $x\in pr(y)$ implies $x\subseteq y$). But then $x\in pr(a_0)$, hence $x\in A$,  which is false. So necessarily $pr(x)\in pr^2(x')$, or $pr(x)\subseteq pr(x')$. By the monotonicity of $pr$, it follows  $x\subseteq x'$.  Interchanging $x$ and $x'$ we have similarly $pr(x')\in pr^2(x)\cup pr^2(a_0)$, from which we infer as before $x'\subseteq x$. Therefore the first equality of this Case implies $x=x'$. By exactly the same argument the second equation  implies $y=y'$.

\vskip 0.1in

{\em Case II.} Let $pr^2(x)\cup pr^2(a_0)=pr^2(y')\cup pr^2(a_1)$ and $pr^2(y)\cup pr^2(a_1)=pr^2(x')\cup pr^2(a_0)$. We  show that this case is impossible.  Consider the first equality. Since $pr(a_0)\in pr^2(a_0)$, it follows    $pr(a_0)\in pr^2(y')\cup pr^2(a_1)$. Then either $pr(a_0)\in pr^2(y')$ or $pr(a_0)\in pr^2(a_1)$. We show that both are impossible. If $pr(a_0)\in pr^2(y')$ then  $pr(a_0)\subseteq pr(y')$. But $pr(a_0)\subseteq A$ while  $pr(y')\subseteq  M$, for all $y'\in M$, so $pr(a_0)\cap pr^2(y')=\emptyset$. If on the other hand $pr(a_0)\in pr^2(a_1)$, we  would have  $pr(a_0)\subseteq  pr(a_1)$. But remember  that we chose $a_0,a_1$ so that $pr(a_0)\not\subseteq  pr(a_1)$ and $pr(a_1)\not\subseteq  pr(a_0)$, a contradiction again. This  proves that the first equality is impossible. This suffices for the impossibility of the entire case, but impossible is also for the same reason the second equality.

\vskip 0.1in

{\em Case III.} This Case is divided into four subcases, all of which are  impossible.

(a) $pr^2(x)\cup pr^2(a_0)=pr^2(y')\cup pr^2(a_1)$, and  the other two unions are submagmas of this. This equality was already proved impossible in Case II.

(b) $pr^2(y)\cup pr^2(a_1)=pr^2(x')\cup pr^2(a_0)$,  and  the other two unions are submagmas of this. This equality also  can be proved impossible exactly as the previous one.

(c) $pr^2(x)\cup pr^2(a_0)=pr^2(x')\cup pr^2(a_0)$  and the other two unions are submagmas of this. It was proved in case I that the first equality implies $x=x'$. However we must have also $pr^2(y)\cup pr^2(a_1)\subseteq pr^2(x)\cup pr^2(a_0)$ and
$pr^2(y')\cup pr^2(a_1)\subseteq pr^2(x')\cup pr^2(a_0)$. However the last two inclusions are impossible too, due again to the incompatibility of $pr^2(a_0)$ and $pr^2(a_1)$. For by the first inclusion we must have $pr(a_1)\in pr^2(x)\cup pr^2(a_0)$. But we saw in Case II that $pr(a_1)\notin pr^2(x)$ and $pr(a_1)\notin  pr^2(a_0)$.

(d) $pr^2(y)\cup pr^2(a_1)=pr^2(y')\cup pr^2(a_1)$  and the other two unions are submagmas of this. This is similar to (c).

In summary, it follows that only case I is possible, and hence  $\langle\langle x,y\rangle\rangle=\langle\langle x',y'\rangle\rangle$ if and only if $x=x'$ and $y=y'$. \telos

\vskip 0.2in

The generalization of magmatic pair to magmatic $n$-tuple, for $n\geq 2$,  can be given exactly as in case of ordinary pairs and tuples. Namely, we set
$$\langle\langle x,y,z\rangle\rangle:=\langle\langle \langle\langle x,y\rangle\rangle,z\rangle\rangle,$$
$$\langle\langle x_1,\ldots,x_n,x_{n+1}\rangle\rangle:=\langle\langle\langle\langle x_1,\ldots,x_n\rangle\rangle,x_{n+1}\rangle\rangle.$$

\begin{Fac} \label{DefOrdPair}
The property  of a magma to be a magmatic ordered pair is definable in $\langle A\cup M,\in\rangle$ with parameters the chosen elements $a_0,a_1\in A$.
\end{Fac}

{\em Proof.} Consider the following property $\phi(x)$ in the language of the structure $\langle A\cup M,\in\rangle$:
$$\phi(x):=x\in M \wedge (\exists y,z\in M)(x=pr(pr^2(y)\cup pr^2(a_0))\cup pr(pr^2(z)\cup pr^2(a_1)).$$
Clearly $\phi(x)$ is an $\{\in\}$-property with parameters $a_0,a_1$ and
$$\mbox{``$x$ is a magmatic  ordered pair''} \Leftrightarrow \langle A\cup M\rangle \models \phi(x).$$ \telos

\section{Magmatic products. Intended vs collateral objects}
In ${\rm ZF}$  binary relations are defined as subsets of a cartesian product $X\times Y$. So we need first to define a magmatic  analogue of this product. Specifically given $x,y\in M$ such that $x,y\subseteq M$, i.e., $x,y\in M\backslash M_1$, we  want the  magmatic  product to be the least magma containing the magmatic  pairs  $\langle\langle z,w\rangle\rangle$ with $z\in x$ and $w\in y$. This magma will be denoted  $x\boxtimes y$. Since whenever  $\langle\langle z,w\rangle\rangle\in x\boxtimes y$ we must have $pr(\langle\langle z,w\rangle\rangle)\subseteq x\boxtimes y$, it is natural to define the magmatic product of $x$ and $y$ as follows:
\begin{equation} \label{E:magmrod}
x\boxtimes y:=\cup\{pr(\langle\langle z,w\rangle\rangle):z\in x \wedge  w\in y\}.
\end{equation}
The question is: what other magmas  are contained in $pr(\langle\langle z,w\rangle\rangle)$ except $\langle\langle z,w\rangle\rangle$ itself? The answer is: $pr(\langle\langle z,w\rangle\rangle)$ contains {\em infinitely many} other magmatic pairs  besides  $\langle\langle z,w\rangle\rangle$, as well as infinitely many other magmas which are not magmatic pairs. Let us first identify which other magmatic pairs are contained in $pr(\langle\langle z,w\rangle\rangle)$. The next Lemma proves these facts.

\begin{Lem} \label{L:pairinpair}
(a) $\langle\langle z',w'\rangle\rangle\in pr(\langle\langle z,w\rangle\rangle)$ iff $\langle\langle z',w'\rangle\rangle\subseteq \langle\langle z,w\rangle\rangle$ iff  $z'\subseteq z$ and $w'\subseteq w$.

(b) In addition, for each  pair $\langle\langle z,w\rangle\rangle$, the following infinite  collections of magmas which are not pairs are contained also in $pr(\langle\langle z,w\rangle\rangle)$:
$$\{pr(pr^2(z')\cup pr^2(a_0)):z'\subseteq z\}\subseteq pr(\langle\langle z,w\rangle\rangle),$$
$$\{pr(pr^2(w')\cup pr^2(a_1)):w'\subseteq w\}\subseteq pr(\langle\langle z,w\rangle\rangle).$$
\end{Lem}

{\em Proof.} (a) Clearly  $\langle\langle z',w'\rangle\rangle\in pr(\langle\langle z,w\rangle\rangle)$ iff  $\langle\langle z',w'\rangle\rangle\subseteq \langle\langle z,w\rangle\rangle$, or
$$pr(pr^2(z')\cup pr^2(a_0))\cup pr(pr^2(w')\cup pr^2(a_1))\subseteq pr(pr^2(z)\cup pr^2(a_0))\cup pr(pr^2(w)\cup pr^2(a_1)).$$
By Lemma \ref{L:two-one} (iv), the last inclusion  holds iff:

(i) $pr^2(z')\cup pr^2(a_0)\subseteq pr^2(z)\cup pr^2(a_0)$ and $pr^2(w')\cup pr^2(a_1)\subseteq pr^2(w)\cup pr^2(a_1)$, or

(ii) $pr^2(z')\cup pr^2(a_0)\subseteq pr^2(z)\cup pr^2(a_0)$ and $pr^2(w')\cup pr^2(a_1)\subseteq pr^2(z)\cup pr^2(a_0)$, or

(iii)  $pr^2(z')\cup pr^2(a_0)\subseteq pr^2(w)\cup pr^2(a_1)$ and $pr^2(w')\cup pr^2(a_1)\subseteq pr^2(w)\cup pr^2(a_1)$, or

(iv) $pr^2(z')\cup pr^2(a_0)\subseteq pr^2(w)\cup pr^2(a_1)$ and $pr^2(w')\cup pr^2(a_1)\subseteq pr^2(z)\cup pr^2(a_0)$.

But it was shown in Theorem \ref{T:pairs} that the second conjunct of (ii), the first conjunct of (iii) and both conjuncts of (iv) are impossible because of the incompatibility of $pr(a_0)$ and $pr(a_1)$. So the claim is equivalent to the  conjunction  (i), of  $pr^2(z')\cup pr^2(a_0)\subseteq pr^2(z)\cup pr^2(a_0)$ and $pr^2(w')\cup pr^2(a_1)\subseteq pr^2(w)\cup pr^2(a_1)$. It  was also shown  in \ref{T:pairs} that this is true  if and only if  $z'\subseteq z$ and $w'\subseteq w$.

(b) Just observe that given  $\langle\langle z,w\rangle\rangle$, for each $z'\subseteq z$ and $w'\subseteq w$, $pr(pr^2(z')\cup pr^2(a_0))\subseteq  \langle\langle z,w\rangle\rangle$ and $pr(pr^2(w')\cup pr^2(a_1))\subseteq  \langle\langle z,w\rangle\rangle$, and the magmas $pr(pr^2(z')\cup pr^2(a_0))$,  $pr(pr^2(w')\cup pr^2(a_1))$ are obviously non-pairs.  \telos

\vskip 0.2in

The question is why we did not define $x\boxtimes y$ simply imitating the standard product and setting $x\boxtimes y:=\{\langle\langle z,w\rangle\rangle:z\in x \wedge  w\in y\}$. The answer is that $\{\langle\langle z,w\rangle\rangle:z\in x \wedge  w\in y\}$ is not a magma, and this follows from Lemma \ref{L:pairinpair} (ii): $\{\langle\langle z,w\rangle\rangle:z\in x \wedge  w\in y\}$ consists of pairs alone while the sets $pr(\langle\langle z,w\rangle\rangle)$ which must be included in $x\boxtimes y$ contain plethora of magmas which are not pairs. Nevertheless, the collection $\{\langle\langle z,w\rangle\rangle:z\in x \wedge  w\in y\}$ deserves a notation and a name because contains precisely the pairs contained in $x\boxtimes y$. We call it {\em weak magmatic product} and we denote it by $\boxdot$. That is, we set
\begin{equation} \label{E:weakprod}
x\boxdot y:=\{\langle\langle z,w\rangle\rangle:z\in x \wedge  w\in y\}.
\end{equation}
Given two collections of magmas $x,y$ let us denote
$x=_p y$ iff  $x$, $y$ contain exactly the same magmatic pairs.

\begin{Fac} \label{F:simply}
For any $x,y\in M\backslash M_1$, $x\boxtimes y=_px\boxdot y$.
\end{Fac}

{\em Proof.} Obviously $x\boxdot y\subseteq x\boxtimes y$. So it suffices to see that every pair of $x\boxtimes y$ belongs to $x\boxdot y$. Pick a  $\langle\langle z',w'\rangle\rangle\in x\boxtimes y=\cup\{pr(\langle\langle z,w\rangle\rangle):z\in x \wedge w\in y\}$. Then $\langle\langle z',w'\rangle\rangle\in pr(\langle\langle z,w\rangle\rangle)$, for some $\langle\langle z,w\rangle\rangle$ with  $z\in x$ and $w\in y$. It follows that  $\langle\langle z',w'\rangle\rangle\subseteq \langle\langle z,w\rangle\rangle$, hence by Lemma \ref{L:pairinpair} (a),  $z'\subseteq z$ and  $w'\subseteq w$. But since $x$ is a magma, $z\in x$ implies $pr(z)\subseteq x$, so  $z'\subseteq z$ implies $z'\in x$. For the same reason $w'\in y$, therefore, by definition, $\langle\langle z',w'\rangle\rangle\in x\boxdot y$. \telos

\vskip 0.2in

Another fact concerning $x\boxtimes y$ is the following.

\begin{Cor} \label{C:pair-produc}
For any $x,y\in M$,  $pr(\langle\langle x,y\rangle\rangle)=_p pr(x)\boxtimes pr(y)$.
\end{Cor}

{\em Proof.}  $\Leftarrow$: Let $\langle\langle u,v\rangle\rangle\in pr(x)\boxtimes pr(y)$. Then, by definition  $\langle\langle u,v\rangle\rangle\in pr(\langle\langle z,w\rangle\rangle)$ for some $z\in pr(x)$ and $w\in pr(y)$, or equivalently, $\langle\langle u,v\rangle\rangle\subseteq  \langle\langle z,w\rangle\rangle$  for some  $z\subseteq x$ and $w\subseteq y$. By Lemma \ref{L:pairinpair}, $\langle\langle u,v\rangle\rangle\subseteq  \langle\langle z,w\rangle\rangle$ holds iff $u\subseteq z$ and $v\subseteq w$. Since further $z\subseteq x$ and $w\subseteq y$, it follows that $u\subseteq x$ and $v\subseteq y$, so by \ref{L:pairinpair} again, $\langle\langle u,v\rangle\rangle\subseteq \langle\langle x,y\rangle\rangle$, or $\langle\langle u,v\rangle\rangle\in pr(\langle\langle x,y\rangle\rangle)$. This proves $\Leftarrow$.

$\Rightarrow$: Let $\langle\langle u,v\rangle\rangle\in pr(\langle\langle x,y\rangle\rangle)$. Then $\langle\langle u,v\rangle\rangle\subseteq \langle\langle x,y\rangle\rangle$, hence by Lemma \ref{L:pairinpair}, $u\subseteq x$ and $v\subseteq y$. It follows $u\in pr(x)$ and $v\in pr(y)$, so $\langle\langle u,v\rangle\rangle\in pr(\langle\langle u,v\rangle\rangle)\subseteq pr(x)\boxtimes pr(y)$. \telos

\vskip 0.2in

The distinction between the collection $x\boxtimes y$, which is a magma, and $x\boxdot y$, which is not a magma but just  a set, spurs a more general distinction  concerning the role of some elements of magmas. In many cases, and especially when we come to define  relations and functions (see next section), we need to focus on some particular {\em isolated elements} of a given magma $x$, which occur also under an  {\em enumeration} $\{y_i:i\in I\}\subset x$, for a set $I$.  The reason that an agent wants to focus on the isolated elements $\{y_i:i\in I\}$ of $x$ is that they play a particular role,  specifically they {\em generate} $x$ in the form
$$\cup\{pr(y_i):i\in I\},$$
that is, $x$ is the least magma containing all $y_i$, $i\in I$. This happens, for example, when an agent wants to match an element $z_i$ of a magma $x$ with  an element $w_i$ of another magma $y$, in order to define a binary relation between $x$ and $y$.  This is done by gathering together the magmatic pairs $\langle\langle z_i,w_i\rangle\rangle$ and defining this relation as the magma
$$\cup\{pr(\langle\langle z_i,w_i\rangle\rangle):i\in I\}.$$
Now just as with the difference between $x\boxtimes y$ and $x\boxdot y$, the difference between  $\cup\{pr(y_i):i\in I\}$ and $\{y_i:i\in I\}$, or between $\cup\{pr(\langle\langle z_i,w_i\rangle\rangle):i\in I\}$ and $\{\langle\langle z_i,w_i\rangle\rangle:i\in I\}$, lies in the fact that the second collections of these couples contain exactly the isolated elements on which we want to focus, and which we can  reasonably  call {\em intended} elements, while the corresponding magmas contain many more elements, namely the elements of
$$\cup\{pr(y_i):i\in I\}\backslash \{y_i:i\in I\}$$ and
$$\cup\{pr(\langle\langle z_i,w_i\rangle\rangle):i\in I\}\backslash \{\langle\langle z_i,w_i\rangle\rangle:i\in I\},$$
respectively, which we call {\em unintended} or {\em collateral} elements. They are simply the elements which {\em depend} on the intended elements (under the dependence relation $\subseteq$ between magmas) and come into play without  agent's intention.

The problem however  is that sets like $\{y_i:i\in I\}$ and $\{\langle\langle z_i,w_i\rangle\rangle:i\in I\}$ are not available in $M$, that is, for  an agent ``living'' in $M$ simply do not exist, and any discussion about them can be carried only externally, by agents living outside of $M$.

To remedy this situation, we shall make here a rather harmless  expansion of $M$ by adjoining to it all such sets of isolated elements. The sets $\{y_i:i\in I\}$ are elements of the class of functions $M^I=\{f|f:I\rightarrow M\}$, where we identify $\{y_i:i\in I\}$ with the function $f:I\rightarrow M$ with $f(i)=y_i$ for each $i\in I$. Then  setting
$$M^{\rm seq}=\cup\{M^I: I\in V\} \ \mbox{and} \ \widehat{M}=M\cup M^{\rm seq},$$
$\widehat{M}$ is the required expansion of $M$ in which we can talk about sets of isolated magmas. $\widehat{M}$ is equipped with $\in$ too, just as $M$.

\begin{Def} \label{D:generate}
{\em A magma $x$ is said to be} generated {\em  by the set of magmas $\{y_i:i\in I\}$ if $x=\cup\{pr(y_i):i\in I\}$. In such a case we call  $y_i$} intended elements of $x$ {\em  and we call  the rest elements of $x$} collateral. {\em  We  denote $int(x)$ and $col(x)$ the sets of intended and collateral elements of $x$, respectively.}
\end{Def}
For the rest of this section we shall work in the structure $\langle \widehat{M},\in\rangle$ rather than $\langle M,\in\rangle$. As already said, the distinction between intended and collateral elements is vividly exemplified in the distinction between the  magmatic  product $x\boxtimes y$ defined in (\ref{E:magmrod}) above, and the corresponding weak product $x\boxdot y$ defined in (\ref{E:weakprod}). It is immediate from the definitions that
$$int(x\boxtimes y)=x\boxdot y, \ \mbox{and} \
col(x\boxtimes y)=(x\boxtimes y)\backslash x\boxdot y.$$
Notice that $x\boxdot y\in \widehat{M}\backslash  M$. Notice also that although a collection of magmas $\{y_i:i\in I\}$ always generates uniquely a magma $x$ with $int(x)=\{y_i:i\in I\}$, the converse is not true. That is, given in general a magma $x$, by no means there is a {\em unique} set of generators $\{y_i:i\in I\}$ for $x$.  In general there can be several sets of generators for a given magma. In fact which subset of $x$ we choose at each case to consider as ``set of generators'' for $x$ is mostly a {\em subjective} choice and decision. It depends on the specific form of $x$ and the purpose for which we want to deal with $x$. For example the magmatic product $x\boxtimes y$ is not a general magma that came up suddenly out of the blue,  but we intentionally and  {\em on purpose} defined it in order to imitate the cartesian product. Then its definition (\ref{E:magmrod}) naturally gives to the elements of $x\boxdot y$ the status of generators.

\vskip 0.2in

\textbf{Remarks about} $\boldsymbol{M^{\rm seq}}$.

1) If we drop the enumeration from a set of magmas $\{y_i:i\in I\}$, we get just a nonempty subset of $M$ (the range of the function $f:I\rightarrow M$, with $f(i)=y_i$,  which produces the particular enumeration). So alternatively $M^{\rm seq}$ might be defined as the set of nonempty subsets of $M$, ${\cal P}(M)\backslash\{\emptyset\}$. However in such a case we should most probably need to use new variables in order to refer to the elements of $M^{\rm seq}$ within the expanded structure $\widehat{M}$. So the use of the elements of $M^{\rm seq}$ as enumerated sets of magmas is rather more convenient.

2) The sets in  $M^{\rm seq}$ comprise also the elements of $M$, since every  $x\in M$, being a subset of $M$, can obviously acquire several enumerations in the context of $V(A)$. So actually $M\subseteq M^{\rm seq}$, although in practice an agent of $\widehat{M}$ aims to enumerate isolated elements of $M$ only, i.e., elements which do not form a magma. Nevertheless the fact that $M\subseteq M^{\rm seq}$ causes no problem in the treatment of $\widehat{M}$.

3) On the other hand, $M^{\rm seq}$ is a ``flat'' class of sets, that is it  does not contain sets of subsets of $M$, e.g. for $x\in M$ $\{x\}\in M^{\rm seq}$, but $\{\{x\}\}\notin M^{\rm seq}$, and hence $\langle x,y\rangle=\{\{x\},\{x,y\}\}\notin M^{\rm seq}$ too. Therefore standard pairs are not available in $\widehat{M}$, which means that even in $\widehat{M}$ it is necessary to define   magmatic ordered pairs as we already did in the last section.

\vskip 0.2in

The role of collateral elements is particularly crucial since they are the main source of problems in the treatment of magmas. Even when  we are interested in a {\em single intended} magma $x$ which belongs to  a magma $y$, in order to deal with it in the context of $y$, we need to treat it as an element of the smallest magma $pr(x)\subseteq y$. But as we already noticed in footnote 1, each magma $pr(x)$ is always  a (local) infinite powerset, and hence a notoriously  unmanageable and uncontrollable collection. That is, together with the single intended element we need to cope with the infinity of collateral elements of $pr(x)\backslash\{x\}$ which are just the elements that depend to $x$.

Now it is interesting that in many cases there are two types of collateral elements. One type of collateral elements of $pr(x)$ may consist of elements $x'\in pr(x)$ which are ``similar'' to $x$ with respect to some specific notion of similarity that holds in the given  case. The other type of collateral elements of $pr(x)$ are those not similar to $x$.  For example if the intended magma is a pair $\langle\langle x,y\rangle\rangle$, the elements of $pr(\langle\langle x,y\rangle\rangle)$  similar to $\langle\langle x,y\rangle\rangle$ are  pairs again. And as we saw above such pairs are exactly those of the form $\langle\langle x',y'\rangle\rangle$, with $x'\subseteq x$ and $y'\subseteq y$. The other type of collateral elements of $pr(\langle\langle x,y\rangle\rangle)$ consists of elements which are not pairs. Lemma \ref{L:pairinpair} proves exactly that: namely each $pr(\langle\langle x,y\rangle\rangle)$ contains both an infinity of similar collateral elements (clause (a)), and an infinity of non-similar collateral elements (clause (b)). Notice that the product  $x\boxtimes y$ and the weak product $x\boxdot y$ differ only in {\em non-similar} elements, i.e.,  $x\boxtimes y$ and $x\boxdot y$  contain the same pairs, and  $x\boxtimes y\backslash x\boxdot y$ consists of non-pairs alone. So, in a sense, the ``noise'' added to $x\boxtimes y$ by collateral elements can be rather  neglected as it comes from elements of the second type. We shall see however that if we come to define magmatic relations and functions in a fashion similar to the standard one, the ``noise'' added may be due to collateral elements of the first kind, i.e., pairs again.

\subsection{Magmatic relations and functions in $\widehat{M}$}
We shall follow the standard method  of defining binary relations over magmas $x,y$ as certain submagmas of the product $x\boxtimes y$. However  not any such submagma $z\subseteq x\boxtimes y$ is eligible for that purpose. Obviously it must contain magmatic pairs, while it follows from  clause (b) of Lemma \ref{L:pairinpair} that  $x\boxtimes y$ may contain submagmas containing no pairs at all. In any case, the most natural way to define such submagmas is by means of a set of generators, which are in fact the {\em intended} pairs which we want to include in the relation, and will be  presented  under an enumeration $\{\langle\langle z_i,w_i\rangle\rangle:i\in I\}$. So a binary relation on $x\boxtimes y$, will be a magma $\textbf{R}\subseteq x\boxtimes y$, such that
\begin{equation} \label{D:relation}
\textbf{R}=\cup\{pr(\langle\langle z_i,w_i\rangle\rangle):i\in I\},
\end{equation}
and the set $\{\langle\langle z_i,w_i\rangle\rangle:i\in I\}$ is the set of intended pairs of $\textbf{R}$. Note that the (important) difference of such a general relation from the product $x\boxtimes y$ (which is a relation too), is  the fact that while the intended pairs of $x\boxtimes y$ are {\em all} pairs $\langle\langle z,w\rangle\rangle$, with $z\in x$ and $w\in y$, i.e., all elements of $x\boxdot y$, the intended pairs of $\textbf{R}$ form only a  subset of $x\boxdot y$. Given such a  binary relation $\textbf{R}$, the {\em domain} and {\em range} of $\textbf{R}$ are the magmas $dom(\textbf{R})$, $ran(\textbf{R})$ defined as follows:
$$dom(\textbf{R})=\cup\{pr(z):(\exists w)(pr(\langle\langle z,w\rangle\rangle)\subseteq \textbf{R})\},$$
$$ran(\textbf{R})=\cup\{pr(w):(\exists z)(pr(\langle\langle z,w\rangle\rangle)\subseteq \textbf{R})\}.$$
First it is not hard to verify the following.

\begin{Fac} \label{F:same}
If $\textbf{R}=\cup\{pr(\langle\langle z_i,w_i\rangle\rangle):i\in I\}$, then
$dom(\textbf{R})=\cup\{pr(z_i):i\in I\}$  and
$ran(\textbf{R})=\cup\{pr(w_i):i\in I\}.$
\end{Fac}

{\em Proof.} We show only the claim for $dom(\textbf{R})$.  First it is immediate from the definition of $dom(\textbf{R})$ that $\cup\{pr(z_i):i\in I\}\subseteq dom(\textbf{R})$. For the converse pick a  $z'\in dom(\textbf{R})$. By the definition of the latter there are $z,w$ such that $z'\in pr(z)$ and $\langle\langle z,w\rangle\rangle\in \textbf{R}$. But also by the definition of $\textbf{R}$ it follows that for some $i\in I$, $\langle\langle z,w\rangle\rangle\subseteq \langle\langle z_i,w_i\rangle\rangle$, from which we get $z\subseteq z_i$, and  since $z'\subseteq z$ we have also $z'\subseteq z_i$, i.e. $z'\in pr(z_i)$ as required. \telos

\vskip 0.2in

The difference between $\textbf{R}$ and $x\boxtimes y$  that makes the treatment of  the former  harder  is that  $\textbf{R}$ contains collateral elements of the {\em first type}, i.e. pairs too, while, as we have seen, $x\boxtimes y$ contains  collateral elements only of the second type, i.e. non-pairs. Specifically, whereas the intended elements of $\textbf{R}$ are precisely the pairs $\langle\langle z_i,w_i\rangle\rangle$, $i\in I$, in fact $\textbf{R}$ contains in addition all pairs of the form $\langle\langle z'_i,w'_i\rangle\rangle$, for $i\in I$, with  $z'_i\subseteq z_i$ and $w'_i\subseteq w_i$. Practically it means that if we want  $\textbf{R}$ to be, say,  a magmatic order relation $\boldsymbol{<}$, with intended domain $\{z_i:i\in I\}$ and intended range $\{w_i:i\in I\}$, such that $z_i\boldsymbol{<}w_i$ and {\em nothing beyond them},  then, in view of  definition (\ref{D:relation}) and the corresponding domains and range  $dom(\textbf{R})$, $ran(\textbf{R})$,  we have:

$\bullet$ to accept as real domain of $\boldsymbol{<}$ the  magma $dom(\boldsymbol{<})=\cup\{pr(z_i:i\in I\}$, and as real range of $\boldsymbol{<}$ the  magma $ran(\boldsymbol{<})=\cup\{pr(w_i:i\in I\}$, and moreover

$\bullet$ to accept  that for {\em every}  $z'_i\subseteq z_i$, and  {\em every}  $w'_i\subseteq w_i$,  $z'_i\boldsymbol{<}w'_i$

\noindent All these extra pairs $\langle\langle z'_i,w'_i\rangle\rangle$ which are forced into $\boldsymbol{<}$,  are regular pairs entirely similar to the intended ones, and so  by no means could be ignored, as possibly could be the case with  collateral elements of second type. We can  see however that in the case of the relation $<$, the collateral pairs  preserve  certain  properties of the  intended ones like  symmetry and transitivity, but fail to have  the antisymmetry property.

\vskip 0.1in

{\textbf{Functions.} The problem with the collateral pairs which are forced into the domain and range of a relation is even more  acute when the relation we want to define and deal with is a function. The peculiarity of functions, compared to  other   binary relations, lies in the critical role that  plays in its definition  the requirement of {\em uniqueness}: a (usual) relation $R$ is a function if and only if for every $x\in dom(R)$ there is a {\em unique} $y\in ran(R)$ such that $\langle x,y\rangle\in R$. There is also another  less critical uniqueness requirement that concerns the elements of $dom(R)$ for a certain kind of functions: for every $y\in ran(R)$ there is a {\em unique} $x\in ran(R)$ such that $\langle x,y\rangle\in R$.  These are the 1-1 functions.  No other kind of binary relation shares such  requirements.

It is clear that the above first  uniqueness condition should apply also to a  relation $\textbf{R}$ which is going to play the role of a ``magmatic function'', that is to each $z\in dom(\textbf{R})$ there must correspond a unique $w\in ran(\textbf{R})$.  From this it follows immediately that such an $\textbf{R}$  should be generated by a set of intended pairs $\langle\langle z_i,w_i\rangle\rangle$, $i\in I$, with the property
\begin{equation} \label{E:unique}
(\forall  i,j\in I) (z_i=z_j \Rightarrow w_i=w_j).
\end{equation}
Obviously (\ref{E:unique}) is a necessary condition in order for $\textbf{R}$  to behave as a function but by no means a sufficient one. Nevertheless, relations satisfying  (\ref{E:unique}) deserve  a name.

\begin{Def} \label{D:necessary}
{\em A magmatic relation $\textbf{R}$ is said to be a} (magmatic) semi-function {\em if
$$\textbf{R}=\cup\{pr(\langle\langle z_i,w_i\rangle\rangle):i\in I\},$$
where the set of generators $\{\langle\langle z_i,w_i\rangle\rangle:i\in I\}$ satisfies (\ref{E:unique}).}
\end{Def}
Now as it happens with any magmatic relation, as we saw earlier,  the unavoidable  occurrence in semi-functions $\textbf{R}$ of the collateral pairs $\langle\langle z'_i,w'_i\rangle\rangle$, for all $z'_i\subseteq z_i$ and $w'_i\subseteq w_i$, spoils completely any sense of uniqueness among all pairs of $\textbf{R}$,  since, simply,  for each $i\in I$, $\{\langle\langle z_i,w'_i\rangle\rangle:w'_i\subseteq w_i\}\subseteq \textbf{R}$, which means that to each $z_i\in dom(\textbf{R})$ there correspond infinitely many ``images'' $w'_i\subseteq w_i$.

In order to circumvent this inevitable fact, we shall follow a different path in order to define  magmatic functions. We shall exploit a remarkable property that have essentially all intended pairs that generate relations: namely if $\langle\langle z,w\rangle\rangle$ is an intended pair of $\textbf{R}$, then both $z$ and $w$ are {\em maximal} elements with respect to the elements $z'$ and  $w'$ which occur in collateral pairs $\langle\langle z',w'\rangle\rangle$, since $z'\subseteq z$ and  $w'\subseteq w$ for all such pairs. We shall use the maximality of the second element of those intended pairs as follows. Given  $\textbf{R}$, let us set   for every $z\in dom(\textbf{R})$,
$$\textbf{R}[z]=\{w:\langle\langle z,w\rangle\rangle\in \textbf{R}\}=\cup\{pr(w):\langle\langle z,w\rangle\rangle\in \textbf{R}\}.$$
[Clearly  $\textbf{R}[z]$ is a magma. Concerning the equality of the two collections, inclusion $\subseteq$ is obvious. For the converse, let $\langle\langle z,w\rangle\rangle\in \textbf{R}$ and  $w'\in pr(w)$. Then  $w'\ \subseteq w$, so  $\langle\langle z,w'\rangle\rangle\subseteq \langle\langle z,w\rangle\rangle$ and hence  $\langle\langle z,w'\rangle\rangle\in \textbf{R}$ since $\textbf{R}$ is a magma. Therefore $w'\in \textbf{R}[z]$.]

Now the uniqueness condition for $\textbf{R}$ could be fulfilled if for each $z\in dom(\textbf{R}$) $\textbf{R}[z]$ contains a {\em greatest} element. In such a case we could define this greatest element to be the  value of $\textbf{R}$ at $z$ and denote it  $\textbf{R}(z)$.

\begin{Def} \label{D:greatest}
{\em A magmatic semi-function  $\textbf{R}$ is said to be a} magmatic function, {\em  or just a} function,  {\em if for every $z\in dom(\textbf{R})$ there is a greatest $w\in ran(\textbf{R})$ such that $\langle\langle z,w\rangle\rangle\in \textbf{R}$, i.e., a $w$ such that for every $w'\in \textbf{R}[z]$, $w'\subseteq w$. This greatest $w$ will be the} value of $\textbf{R}$ at $z$ {\em and will be denoted $\textbf{R}(z)$. We use symbols $\textbf{F}$, $\textbf{G}$ etc, to denote magmatic functions.}
\end{Def}
The simplest examples of  magmatic functions are the relations  generated by a single pair,  i.e.   of the form $\textbf{F}=pr(\langle\langle z,w\rangle\rangle)$   (recall the by Corollary  \ref{C:pair-produc}, $pr(\langle\langle z,w\rangle\rangle$ is also identical to the product $pr(z)\boxtimes pr(w)$). Then every $z'\in dom(\textbf{F})$ is related to $w$ and obviously $w$ is greatest with this property, so for every $z'\in dom(\textbf{F})$, $\textbf{F}(z')=w$. On the other hand if $\textbf{R}=pr(\langle\langle z,w_1\rangle\rangle)\cup pr(\langle\langle z,w_2\rangle\rangle)$ and $w_1$, $w_2$ are incomparable, i.e., $w_1\not\subseteq w_2$ and $w_2\not\subseteq w_1$, then $\textbf{R}$ is not a function, since $w_1,w_2\in  \textbf{R}[z]$ but there is no $w\in \textbf{R}[z]$ such that $w_1,w_2\subseteq w$. More generally, if $\textbf{R}=pr(\langle\langle z_1,w_1\rangle\rangle)\cup pr(\langle\langle z_2,w_2\rangle\rangle)$, where  $w_1$, $w_2$ are incomparable and $z_1\cap z_2\neq\emptyset$, $\textbf{R}$ is not a function. For if  $z_1\cap z_2\neq\emptyset$ then the latter is a magma\footnote{Recall that for two magmas $z_1, z_2$, $z_1\cap z_2$ can be empty, in which case  $z_1\cap z_2\notin M$. For example $pr(a_0)\cap pr(a_1)=\emptyset$ if there is no $b$ such that $b\preccurlyeq a_0$ and $b\preccurlyeq a_1$. It follows also by induction that for such $a_0,a_1$   $pr^n(a_0)\cap pr^n(a_1)=\emptyset$ for all  $n\geq 1$. On the other hand, if $z_1\cap z_2\neq\emptyset$ then it is  a magma because for each $w\in z_1\cap z_2$, $pr(w)\subseteq z_1\cap z_2$. }  and since $z_1\cap z_2\subseteq z_1,z_2$, it follows that both  $\langle\langle z_1\cap z_2,w_1\rangle\rangle)$ and  $\langle\langle z_1\cap z_2,w_2\rangle\rangle)$ are contained in $\textbf{R}$, so $w_1,w_2\in \textbf{R}[z_1\cap z_2]$, and since $w_1,w_2$ are incomparable, $\textbf{R}[z_1\cap z_2]$  does not contain a greatest element.

A sufficient condition for a semi-function to be a function is the following.

\begin{Lem} \label{L:sufficient}
Let $\textbf{R}$ be a semi-function generated by $\langle\langle z_i,w_i\rangle\rangle$, $i\in I$. If the set $\{z_i:i\in I\}$ is pairwise disjoint, then  $\textbf{R}$ is a function.
\end{Lem}

{\em Proof.}  Let $\textbf{R}$ have the mentioned property. We have to show that for every $z\in dom(\textbf{R})$, $\textbf{R}[z]$ has a greatest element. Now if $\langle\langle z,w\rangle\rangle$ is any pair of $\textbf{R}$, necessarily $\langle\langle z,w\rangle\rangle\subseteq \langle\langle z_i,w_i\rangle\rangle$ for a  unique $i$. Because if  $\langle\langle z,w\rangle\rangle$ is a subset of both $\langle\langle z_i,w_i\rangle\rangle$ and $\langle\langle z_j,w_j\rangle\rangle$, with  $i\neq j$, we should have $z\subseteq z_i\cap z_j$, which is false since by assumption $z_i,z_j$ are disjoint. It follows that for every $\langle\langle z,w\rangle\rangle\in \textbf{R}$, the greatest element of $\textbf{R}[z]$ is the (unique) $w_i$ for which  $\langle\langle z,w\rangle\rangle\subseteq \langle\langle z_i,w_i\rangle\rangle$. \telos

\vskip 0.2in

Notice however that the condition of Lemma \ref{L:sufficient} is not necessary in order for $\textbf{R}$ to be a function. The simplest counterexample is any semi-function $\textbf{R}$ which is generated by any  set of pairs of the form $\{\langle\langle z_i,w\rangle\rangle:i\in I\}$, i.e., such that $ran(\textbf{R})$ is generated by $\{w\}$.  A less trivial example is the following.

\begin{Ex} \label{E:counter}
Let $\textbf{R}=pr(\langle\langle z_1,w_1\rangle\rangle)\cup pr(\langle\langle z_2,w_2\rangle\rangle)$ where:

(1) $z_1,z_2$ are incomparable, (2) $z_1\cap z_2\neq \emptyset$,  and (3) $w_1\varsubsetneq w_2$.

\noindent Then $\textbf{R}$ is a function.
\end{Ex}

The  verification of \ref{E:counter} is left to the reader. \\

[Hint: Let $\langle\langle z,w\rangle\rangle\in \textbf{R}$ be  any pair. Consider the cases: \\
(a) $\langle\langle z,w\rangle\rangle\in pr(\langle\langle z_1,w_1\rangle\rangle)$ and $\langle\langle z,w\rangle\rangle\in pr(\langle\langle z_2,w_2\rangle\rangle)$. \\
(b) $\langle\langle z,w\rangle\rangle\in pr(\langle\langle z_1,w_1\rangle\rangle)$ and $\langle\langle z,w\rangle\rangle\notin pr(\langle\langle z_2,w_2\rangle\rangle)$. \\
(c) $\langle\langle z,w\rangle\rangle\in pr(\langle\langle z_2,w_2\rangle\rangle)$ and $\langle\langle z,w\rangle\rangle\notin pr(\langle\langle z_1,w_1\rangle\rangle)$.]

Although the condition of Lemma \ref{L:sufficient} (the disjointness of the magmas that generate the domain of $\textbf{R}$) is very restrictive, we shall see  in the next section that it can actually be  used in some cases.

It is clear that in  the vast majority of situations, semi-functions are not functions, and  the real source of the problem lies in  the  {\em way}  magmatic pairs were defined in section 3, which  gives rise to the existence of a multitude, actually an infinity, of sub-pairs to any given pair $\langle\langle x,y\rangle\rangle$, namely to $\langle\langle x',y'\rangle\rangle$ for any $x'\subseteq x$ and $y'\subseteq y$.  It is exactly  this property which causes  the occurrence of dependent  unintended (collateral) pairs inside magmatic relations.  So the question is:

\vskip 0.1in

(\dag)  Is it  possible to define magmatic pairs so that for a given $\langle\langle x,y\rangle\rangle$, there can be no pair $\langle\langle x',y'\rangle\rangle$ with $x'\neq x$ or $y'\neq y$ such that  $\langle\langle x',y'\rangle\rangle\subseteq \langle\langle x,y\rangle\rangle$?

\vskip 0.1in

I strongly guess that the answer to question (\dag) is negative.

\section{Magmatic  natural and ordinal numbers in $\widehat{M}$}
The primitive constituent for the construction of natural and ordinal numbers in ${\rm ZF}$ (according to von  Neumann's definition)  is $\emptyset$. Identifying $\emptyset$ with $0$ we define  the rest numbers setting $1=\{0\}$, $2=\{0,1\}$, etc.  In the absence of $\emptyset$ in $M$, we shall start with an  atom $a_0\in A$, which we choose once and for all and keep it fixed  henceforth. The element  $a_0$ is the starting point but not exactly the ``analogue'' of $\emptyset$ and $0$.   Below we shall denote magmatic natural numbers by boldface versions $\textbf{0}, \textbf{1},\ldots,\textbf{n},\ldots$ of the corresponding usual numbers $0,1,\ldots,n,\ldots$, and more generally magmatic  ordinals  by boldface versions $\boldsymbol{\omega, \alpha,\beta,}\ldots$ of the corresponding usual  ordinals $\omega,\alpha,\beta,\ldots$. We denote by $\textbf{n}'$ and $\boldsymbol{\alpha'}$   the successors of $\textbf{n}$ and $\boldsymbol{\alpha}$, respectively.

At first thought  the  magmatic number  $\textbf{0}$ should be defined as the simplest magma containing $a_0$,  $pr(a_0)$, just as $0$ is defined as $\emptyset$.  However  representing $\textbf{0}$ by $pr(a_0)$ would not allow us to define a greater number  $\textbf{n}$ so as to include $\textbf{0}$ as a submagma, in  the form  $\textbf{n}=\textbf{0}\cup x=pr(a_0)\cup x$, where $x\in M\backslash M_1$.  Because in such a case, by Lemma \ref{L:legitim} (iii), $\textbf{n}\notin M$  since   $pr(a_0)\in M_1$ while $\boldsymbol{n}\backslash pr(a_0)$ would belong to $M\backslash M_1$. So we shall set instead $\textbf{0}=pr^2(a_0)=pr(pr(a_0))$.  Concerning the definition of the successors $\textbf{n}'$ and $\boldsymbol{\alpha'}$, we imitate  the standard definition  $n+1=n\cup\{n\}$ for natural numbers and set $\textbf{n}'=\textbf{n}\cup pr(\textbf{n})$, and similarly for ordinals, since  the least magma $pr(\textbf{n})$ that contains  $\textbf{n}$ is the analogue of   the singleton $\{n\}$.   So  we define $\textbf{n}$ and $\boldsymbol{\alpha}$ in ${\rm ZFA}$ by  induction on $n\in \N$ and $\alpha\in Ord$, respectively, as follows:

\vskip 0.1in

$\textbf{0}:=pr^2(a_0)$,

$\textbf{1}:=\textbf{0}\cup pr(\textbf{0})$,

$\textbf{2}:=\textbf{1}\cup pr(\textbf{1})$,

$\textbf{n+1}:=\textbf{n}\cup pr(\textbf{n})$.

\vskip 0.1in

\noindent Further for the infinite  magmatic ordinals we set:

\vskip 0.1in

$\boldsymbol{\omega}:=\cup\{\textbf{n}:n<\omega\}$,

$\boldsymbol{\alpha+1}:=\boldsymbol{\alpha}\cup pr(\boldsymbol{\alpha})$,

$\boldsymbol{\alpha}:=\cup\{\boldsymbol{\beta}:\beta<\alpha\}$, for limit $\alpha$.

\vskip 0.1in

By the same token, extending the last definition, we define the addition of magmatic ordinals by setting:

\vskip 0.1in

$\boldsymbol{\alpha+0}=\boldsymbol{\alpha}$.

$\boldsymbol{\alpha + (\beta+1)}=\boldsymbol{(\alpha+\beta)+1}:=\boldsymbol{(\alpha+\beta)}\cup pr(\boldsymbol{\alpha+\beta})$ and

$\boldsymbol{\alpha + \beta}:=\cup\{\boldsymbol{\alpha + \gamma}:\gamma<\beta\}$, for limit $\beta$.

\vskip 0.1in

(Multiplication $\boldsymbol{\alpha\cdot\beta}$ is defined as well by the help of the just defined  $\boldsymbol{\alpha+\beta}$, exactly as  with standard ordinals, but we shall not need it below).

Let $\textbf{N}=\{\textbf{n}:n\in\N\}$ and $\textbf{Ord}:=\{\boldsymbol{\alpha}:\alpha\geq 1\}$. Clearly $\textbf{N}\subset M$, and $\textbf{Ord}\subset M$, the latter being a proper class. Also    $\textbf{N}\notin M$ but $\textbf{N}\in M^{\rm seq}$, and hence $\textbf{N}\in{\widehat  M}$, because it is the range of the function $f\in M^{\N}$ such that $f(n)=\textbf{n}$.

\vskip 0.1in

As is well-known, the order relation of ordinals is $\in$, i.e., $\alpha<\beta\Leftrightarrow \alpha\in\beta$. Concerning the magmatic ordinals, the corresponding relation is the proper inclusion $\subset$. Namely the following holds.

\begin{Lem} \label{L:ordering}
For all $\alpha,\beta\in Ord$, $\alpha<\beta \Leftrightarrow \boldsymbol{\alpha}\subset \boldsymbol{\beta}$.
\end{Lem}

{\em Proof.} $\Rightarrow$: As is well-known  $\alpha<\beta$ iff there is a $\gamma>0$ such that $\beta=\alpha+\gamma$, so it suffices to show that for all $\alpha, \beta$ with $\beta>0$, $\boldsymbol{\alpha}\subset \boldsymbol{\alpha+\beta}$. This is shown by an easy induction using the clauses above that define  $\boldsymbol{\alpha+\beta}$. That is (i) $\boldsymbol{\alpha}\subset \boldsymbol{\alpha+1}$, (ii) $\boldsymbol{\alpha}\subset \boldsymbol{\beta}\Rightarrow \boldsymbol{\alpha}\subset \boldsymbol{\beta+1}$, and (iii) $(\forall \gamma<\beta)(\boldsymbol{\alpha}\subseteq \boldsymbol{\gamma})\Rightarrow \boldsymbol{\alpha}\subset \boldsymbol{\beta}$, for limit $\beta$.

$\Leftarrow$: Let $\boldsymbol{\alpha}\subset \boldsymbol{\beta}$. If we unfold the inductive definition of $\boldsymbol{\beta}$, we see that it has a unique  expansion     as  the  union $\boldsymbol{\beta}=\cup\{x_\xi:\xi\leq\beta\}$, if $\beta$ is successor, or $\boldsymbol{\beta}=\cup\{x_\xi:\xi<\beta\}$, if $\beta$ is limit, where:

(i) $x_0=pr^2(a_0)$,

(ii)  $x_{\xi+1}=pr(\cup\{x_\zeta:\zeta\leq \xi\})$ and

(iii) $x_\xi=pr(\cup\{x_\zeta:\zeta<\xi\}$, if $\xi$ is limit. \\
Moreover, in order for a sub-magma $y\subset \boldsymbol{\beta}$ to be a magmatic ordinal  it is necessary and sufficient that for some $\alpha<\beta$, $y$ is the union of the $\alpha$-initial segment of $(x_\xi)_\xi$, to be precise  $y=\cup\{x_\xi:\xi\leq \alpha\}$, if $\alpha$ is successor,  or $y=\cup\{x_\xi:\xi<\alpha\}$, if $\alpha$ is limit. In addition in this case clearly  $y=\boldsymbol{\alpha}$. Since $\alpha<\beta$, this proves that $\boldsymbol{\alpha}\subset \boldsymbol{\beta}\Rightarrow \alpha<\beta$. \telos

\vskip 0.2in

It follows from the previous Lemma that the relation $\subset$ can be used as a definition for the ordering $\boldsymbol{<}$ of magmatic ordinals.

\vskip 0.1in

\textbf{An alternative definition.} An alternative way to define natural and ordinal numbers would be to set for successor ordinals $\boldsymbol{\alpha}'=pr(\alpha)$, instead of $\boldsymbol{\alpha}\cup pr(\boldsymbol{\alpha})$, the other parts of the definition being as before, i.e., $\boldsymbol{0}=pr^2(a_0)$ and $\boldsymbol{\alpha}=\cup\{\boldsymbol{\beta}:\beta<\alpha\}$, for limit $\alpha$. In particular, for magmatic natural numbers we set:

\vskip 0.1in

$\textbf{0}:=pr^2(a_0)$,

$\textbf{1}:=pr(\textbf{1})=pr^3(a_0)$,

$\textbf{n}':=pr(\textbf{n})$, so, by induction on $\N$,

$\textbf{n}=pr^{n+2}(a_0)$ for all $n\in\N$.

\vskip 0.1in

Note that this definition is analogous to the alternative definition of standard natural numbers according to which  $0=\emptyset$, $1=\{0\}$, $2=\{1\}=\{\{0\}\}$,  and generally $n+1=\{n\}$, i.e., $n=\{\{\cdots\{0\}\cdots\}\}$ for all $n\in\N$, which is used by many authors.

Let $\N^*$ be the set (in $M^{\rm seq}$) of these alternative natural  numbers.
An ``advantage'' of $\N^*$ over $\N$  for our context is that it allows us to define countably generated magmas and show that they are ranges of functions with domain $\cup\{pr(\textbf{n}):\textbf{n}\in \N^*\}$.  Fist we have the following.

\begin{Fac} \label{F:disjoint}
For all $\textbf{m}\neq \textbf{n}\in \N^*$,  $\textbf{m}\cap \textbf{n}=\emptyset$.
\end{Fac}

{\em Proof.}  Recall that $pr(a_0)\in M_1$, so $pr^2(a_0)=\textbf{0}\in M_2$, and hence  $\textbf{0}\subseteq M_1$. By induction we easily see that for every $n\geq 1$, $\textbf{n}=pr^{n+2}(a_0)\subseteq M_{n+1}$. By Lemma 4.8 of \cite{Tz24}, $M_{m+1}\cap M_{n+1}=\emptyset$ for all $m\neq n$, so $\textbf{m}\cap \textbf{n}=\emptyset$. \telos

\vskip 0.2in

Notice that Fact \ref{F:disjoint} is not true for ordinals in general, since in general $\boldsymbol{\alpha}\subseteq M_\alpha$ for every $\alpha$, but the levels $M_\alpha$, $M_\beta$ are not always disjoint. For example  by the definition $\boldsymbol{\omega}=\cup\{\textbf{n}:\textbf{n}\in\N^*\}$,   $\textbf{n}\subseteq\boldsymbol{\omega}$ for every $n$, and more generally, for every   limit ordinal $\alpha$ and every $\beta<\alpha$, $\boldsymbol{\beta}\subseteq \boldsymbol{\alpha}$.

\begin{Def} \label{D:countgen}
{\em A magma $x$ is said to be} countably generated {\em if there is a sequence $\{y_n:n\in\omega\}$ in $M^{\rm seq}$ such that $x=\cup\{pr(y_n):n\in \omega\}$.}
\end{Def}

\begin{Prop} \label{P;countfunc}
If $x\in M$ is countably generated,   there is a magmatic function $\textbf{F}$ such that $dom(\textbf{F})=\boldsymbol{\omega}\backslash\boldsymbol{0}$ and $ran(\textbf{F})=x$.
\end{Prop}

{\em Proof.} Let  $x=\cup\{pr(y_n):n\in\omega\}$ be a countably generated magma, and let
$$\textbf{F}=\cup\{pr(\langle\langle \textbf{n},y_n\rangle\rangle:n\geq 0\}.$$
Obviously $\textbf{F}$ is a semi-function. Since further, by Fact \ref{F:disjoint}, $\{\textbf{n}:n\geq 0\}$ is a pairwise disjoint set, it follows from Lemma \ref{L:sufficient} that  $\textbf{F}$ is a function. By Fact \ref{F:same},
$$ran(\textbf{F})=\cup\{pr(y_n):n\in \omega\}=x.$$
Also, by the definition of $\boldsymbol{\omega}=\cup\{\boldsymbol{n}:n\geq 0\}=\cup\{pr^{n+2}(a_0):n\geq 0\}$, we have   $$dom(\textbf{F})=\cup\{pr(\boldsymbol{n}):n\geq 1\}= \cup\{pr^{n+3}(a_0):n\geq 0\}=\boldsymbol{\omega}\backslash pr^2(a_0)=\boldsymbol{\omega}\backslash\boldsymbol{0}.$$ \telos

\section{The question about  Separation and Replacement in $M$ }
In \cite{Tz24} we showed that some  axioms of ${\rm ZFC}$ are still true in $\langle M,\in\rangle$. Namely, Extensionality and Foundation hold trivially, while Powerset  and a weak form of Union were proven as  nontrivial truths of $M$. Of the rest  axioms,  Emptyset, Pairing, Infinity and Choice  are definitely (and irreparably)  false.  What remains are  {\em Separation} and {\em Replacement}  (abbreviated {\em Sep} and {\em Rep}, respectively). Of them  {\em Sep}  is clearly false in $M$ in its  general form, but we shall see that a restricted form, that is  a scheme where $\phi$ ranges over a special class of formulas, does hold in $M$. On the other hand  there can be no such restricted form of {\em Rep}, a fact  due to the problems with the definition of  functions in $M$ that  we saw  in section 4.1.

For a class $X\subseteq M$ definable by a formula $\phi(x)$, possibly with parameters, that is $X=\{x:M\models  \phi(x)\}$, we write $X=X_{\phi}$. The truth of $Sep$ in $M$ for the formula $\phi(x)$ has the usual meaning, i.e., for any set $u\in M$, $X_\phi\cap u\in M$, provided $X_\phi\cap u\neq \emptyset$, since $\emptyset\not\in M$. To see  that $Rep$ is false in $M$ in its general form, take for example the formula $\phi(x)=(x=y_0)$, for some magma $y_0$ as a parameter, and let $u$ be any magma. Then $X_\phi\cap u=\{x:x=y_0\wedge x\in u\}$. If $y_0\in u$ then $X_\phi\cap u=\{y_0\}\not\in M$. If $y_0\notin u$, then again $X_\phi\cap u=\emptyset\notin M$.

The question is whether {\em Sep} can hold in $M$ for the classes defined by an appropriate class   of properties $\phi(x)$. The following is a simple necessary and sufficient condition in order for {\em Sep} to hold in $M$ for a specific property  $\phi(x)$.

\begin{Lem} \label{L:iff}
{\em Sep}  holds in $M$ for  $\phi(x)$, if and only if for every  $x\in X_{\phi}$, $pr(x)\subseteq X_{\phi}$, i.e.,
\begin{equation} \label{E:condition}
x\in X_{\phi}  \wedge  y\subseteq x \rightarrow y\in X_{\phi}.
\end{equation}
\end{Lem}

{\em Proof.} Assume the condition is true and $u\in M$. We have to show that $u\cap X_{\phi}$ is a magma, i.e., $x\in u\cap X_{\phi}$ implies $pr(x)\subseteq u\cap X_{\phi}$. Pick such a $x\in u\cap X_{\phi}$.  Now $x\in u$ implies $pr(x)\subseteq u$  since $u$ is already a magma, while by our condition $x\in X_{\phi}$ implies $pr(x)\subseteq  X_{\phi}$. So $pr(x)\subseteq u\cap X_{\phi}$, i.e., $u\cap X_{\phi}$ is a magma, so {\em Sep} holds for ${\phi}$. For the converse,  assume that the condition fails, i.e., there is a $x\in X_{\phi}$ and a $y\subseteq x$ such that $y\notin X_{\phi}$. Fix such magmas  $x$ and $y\subseteq x$. Now clearly there is always a magma $u$ such that $x\in u$. Then  $x\in u\cap X_{\phi}$ while $y\notin u\cap X_{\phi}$. It follows that $u\cap X_{\phi}$ is not a magma and  {\em Sep} fails in $M$. \telos

\vskip 0.2in

Lemma \ref{L:iff} leads to the following:

\begin{Def} \label{D:magclass}
{\em A class $X\subseteq M$ is said to be} magmatic {\em if it satisfies condition (\ref{E:condition}). Accordingly, a formula $\phi(x)$ of $L$ is said to be} magmatic {\em if $X_{\phi}$ is magmatic, i.e.,
$$M\models (\forall x,y)(\phi(x) \wedge y\subseteq x \rightarrow \phi(y)).$$}
\end{Def}

The following  is straightforward.

\begin{Cor} \label{C:straight}
A definable class $X\subseteq M$ is magmatic if and only if it is defined by a magmatic formula.
\end{Cor}

There is a simple way to generate all magmatic classes/formulas from the classes of all definable classes/formulas, respectively.  This is done by the following completion procedure.

Given a class $X\subseteq M$ the (magmatic) {\em completion of $X$} is the class
$$\overline{X}=\{x\in M:(\exists y)(y\in X \wedge x\subseteq y)\}.$$
Analogously, given a formula $\phi(x)$ of $L$ in the free variable $x$, possibly with parameters, the (magmatic) {\em  completion of $\phi(x)$} is the formula
$$\overline{\phi}(x):=(\exists y)(\phi(y) \wedge x\subseteq y).$$
The following is easy.

\begin{Fac} \label{F:basics}
For every $\phi$, $\overline{X}_\phi=X_{\overline{\phi}}$.
\end{Fac}

{\em Proof.}  By the definitions above we have
$$x\in \overline{X}_\phi \Leftrightarrow (\exists y)(y\in X_\phi \wedge x\subseteq y)\Leftrightarrow(\exists y)(M\models \phi(y) \wedge x\subseteq y)\Leftrightarrow x\in X_{\overline{\phi}}.$$ \telos

The {\em Magmatic Separation Scheme} (MSS) is the following claim:

\vskip 0.1in

$\rm (MSS)$ \ \  \ For every magmatic formula $\phi(x)$,

\hskip 0.6in $(\forall u)(\exists v)[X_{\phi}\cap u\neq \emptyset \rightarrow X_{\phi}\cap u=v]$.

\vskip 0.1in

\begin{Prop} \label{P:true}
The {\rm MSS} is true in $M$.
\end{Prop}

{\em Proof.} Let $\phi(x)$ be a magmatic formula, $X_\phi=\{x\in M:M\models\phi(x)\}$ and let $u\in M$ such that $X_\phi\cap u\neq\emptyset$. We have to show that $X_\phi\cap u\in M$. We have seen that $\phi$ is magmatic implies $X_\phi$ is a  magmatic class, i.e., for every $x\in X_\phi$ and  $y\subseteq x$, $y\in X_\phi$. Let $X_\phi\cap u=v$.  It suffices to show that $x\in v$ and $y\subseteq x$, implies $y\in v$. But the latter implication holds for both $X_\phi$ and $u$, since $X_\phi$ is magmatic and $u$ is a magma, so indeed $y\in v$. \telos

\vskip 0.2in

In contrast to Separation, Replacement does not seem  to admit an analogous adjustment that would make it true in $M$, and the reason is essentially its  {\em functional} character and the problems that functions present when  represented inside $M$, as we saw in section 4.1.

In  ZF Replacement, like Separation, can also be formulated  in terms of definable classes of ordered pairs.  Call a formula $\phi(x,y)$  {\em functional}, if it has the property  $(\forall x)(\exists! y)(\phi(x,y)$. The corresponding class defined by $\phi(x,y)$, $F=\{\langle x,y\rangle:\phi(x,y)\}$, is called functional too, and  we can write $F(x)=y$ instead of $\phi(x,y)$ or $\langle x,y\rangle\in F$.
Then the {\rm ZF}-Replacement scheme says that for every functional class $F$ which is definable in our universe $V$, and every set  $u\in V$, the image-collection $\{F(x):x\in u\}=\{y:(\exists x)(x\in u \wedge \langle x,y\rangle\in F\}$ of $u$ under $F$ is also a set of $V$.

The  corresponding definition of functional relations and formulas in $M$  reads similarly, but their extension cannot be represented by classes of magmatic pairs $\langle\langle x,y\rangle\rangle$. For example a natural functional relation  defined in $M$ is $pr(x)=y$, as well as its iterates $pr^2(x)$, $pr^3(x)$ etc. But if we try to represent this operation by the corresponding class of magmatic pairs $X=\{\langle x,y\rangle\rangle:pr(x)=y\}$, then $X$ is not magmatic, and also for any magma $u$, the ``restriction'' of $pr$ to $u$,  $v=\{\langle\langle x,y\rangle\rangle: pr(x)=y \wedge x\in u\}$ is  not a magma. This is obvious from the fact that every $\langle\langle x,y\rangle\rangle$ contains  submagmas $w$ which are not pairs, so $w$  cannot belong to $v$. But even if we confine ourselves  only to submagmas of $\langle\langle x,y\rangle\rangle$ which are pairs,  then for all $x'\subset x$,  $\langle\langle x',y\rangle\rangle\subseteq \langle\langle x,y\rangle\rangle$ but   $\langle\langle x',y\rangle\rangle\notin v$, since $pr(x')\neq y$.

Nevertheless, even without the possibility to represent functional relations by magmatic pairs, the question  remains whether the image of a magma under such a relation is a magma. The answer is negative. A trivial example is by taking the relation $\phi(x,y):(x=x \wedge y=y_0$), for a parameter $y_0$, which defines the constant class function $F(x)=y_0$ for all $x$. Then for every magma $u$ the image of $u$ under $F$ is $\{y_0\}$, which is not a magma. A less trivial example is given by the operation $pr$.

\begin{Ex} \label{counter}
Replacement fails in $M$ for the functional relation $pr(x)=y$.
\end{Ex}

{\em Proof.} Assume the contrary, i.e., given a magma $u$, the set $v=\{pr(x):x\in u\}$ is a magma. It means that for every $x\in u$, $pr(pr(x))\subseteq v$. We  shall see that this is false. Fix a $x\in u$. Then for every $z\in pr(pr(x))$, i.e., $z\subseteq pr(x)$, we must have $z\in v$, i.e., $z=pr(y)$ for some $y\in u$. Now in fact {\em some}  magmas $z\subseteq pr(x)$ do have the form $z=pr(y)$ for some $y\in u$, because $u$ is a magma and when $x\in u$ then for every $y\subseteq x$, $y\in u$ too, and $pr(y)\subseteq pr(x)$, due to the monotonicity of $pr$.  However not all $z\subseteq pr(x)$ are of this kind. Specifically, take distinct submagmas $x_1,\ldots, x_n$ of $x$. Then $pr(x_i)\subseteq pr(x)$ for each $i=1,\ldots,n$, and hence  $pr(x_1)\cup\cdots \cup pr(x_n)$ is also a submagma of $pr(x)$, which obviously is not of the form $pr(y)$. \telos

\vskip 0.2in

The only remaining option seems to be  to restrict the defining formulas involved in the $Rep$ scheme. Such restrictions  could not  be different from those applied to the formulas  used in the case of Separation. Namely, given  a functional relation defined  by a formula $\phi(x,y)$, to try some kind of ``magmatic completion'' of $\phi$  analogous to the one defined in Corollary \ref{C:straight} above, say
$$\overline{\phi}(x,y):=(\exists x',y')(\phi(x',y') \wedge \langle\langle x,y\rangle\rangle\subseteq \langle\langle x',y'\rangle\rangle).$$
The problem however is that  $\overline{\phi}(x,y)$ defines  no more a functional relation. For if $\overline{\phi}(x,y)$ holds for two elements $x,y$, then it follows easily from  the definition of $\overline{\phi}$ and the properties  of magmatic pairs that  for {\em any} $y_1\subseteq y$ $\overline{\phi}(x,y_1)$ holds too,  so   $\overline{\phi}$ is not functional.

\end{document}